\documentclass{article}

\usepackage{amsmath}
\usepackage{amssymb}
\usepackage{amsthm}

\usepackage[margin=1in]{geometry}
\theoremstyle{theorem}
\newtheorem{Theorem}{Theorem}[section]

\newtheorem{Lemma}[Theorem]{Lemma}
\newtheorem{Proposition}[Theorem]{Proposition}
\newtheorem{Conjecture}[Theorem]{Conjecture}

\theoremstyle{definition}
\newtheorem{Definition}[Theorem]{Definition}
\newtheorem{Remark}[Theorem]{Remark}

\newcommand{\N}{\mathbb{N}}
\newcommand{\R}{\mathbb{R}}
\newcommand{\Z}{\mathbb{Z}}

\newcommand{\C}{\mathbb{C}}
\newcommand{\E}{\mathbb{E}}
\newcommand{\Hyp}{\mathbb{H}}
\newcommand{\Sph}{\mathbb{S}}

\usepackage{listings}
\usepackage{color}

\lstdefinestyle{base}{
  emptylines=1,
  moredelim=**[is][\color{blue}]{@}{@},
  moredelim=**[is][\color{green}]{@@}{@@},
}

\usepackage{tikz}
\usepackage{tikz-cd}
\usetikzlibrary{decorations.markings}
\usetikzlibrary{arrows,shapes,positioning}
\usetikzlibrary{matrix}
\usetikzlibrary{cd}
\usetikzlibrary{arrows.meta}
\usetikzlibrary{calc}

\usepackage{bbm}
\usepackage{mathrsfs}

\usepackage{scalerel}

\usepackage{multirow}

\usepackage{hyperref}

\usepackage{ltablex} 
\keepXColumns

\begin{document}

\pagenumbering{arabic}

\title{2-Linearizability of Geometric 3-Manifold Groups Over Commutative Rings}
\author{Montek Singh Gill\footnote{Email: montekg@umich.edu}}

\maketitle

\begin{abstract}
The fundamental groups of compact 3-manifolds are known to be residually finite. Feng Luo conjectured that a stronger statement is true, by only allowing finite groups of the form $\mathrm{PGL}(2,R)$, where $R$ is a finite commutative ring. In earlier work, this conjecture was disproven in full generality. The conjecture arose in the context of orientable connected compact 3-manifolds which are geometrizable. By constructing explicit faithful linear representations using rings with nilpotent elements, we demonstrate that the conjecture holds for six of the eight Thurston model geometries, namely all but $\Sph^3$ and $\widetilde{\mathrm{SL}_2}$. In the case of $\Sph^3$, the conjecture holds if we replace $\mathrm{PGL}(2,R)$ with $\mathrm{GL}(2,R)$. A spherical counterexample for the projective variant is the Poincar\'{e} homology sphere $\Sigma(2,3,5)$. In the case of $\widetilde{\mathrm{SL}_2}$, the conjecture fails to hold for both the projective and non-projective variants; a counterexample is provided by the Brieskorn sphere $\Sigma(2,3,7)$.
\end{abstract}

\tableofcontents

\section{Introduction}

In this paper, a $3$-manifold is a smooth $3$-manifold, and a geometric structure on a $3$-manifold, in the sense of Thurston, is a Riemannian metric on $M$ which is locally isometric to one of Thurston's eight three-dimensional geometries: $\E^3, \Sph^3, \Hyp^3, \Sph^2 \times \E, \Hyp^2 \times \E, \mathrm{Nil}, \mathrm{Sol}, \widetilde{\mathrm{SL}_2}$. Given an orientable connected compact $3$-manifold $M$, a hyperbolic structure on $M$ gives rise to a developing map
\[
\mathrm{dev} \colon \widetilde M \to \Hyp^3
\]
which is an isometry from the universal cover of $M$ to three-dimensional hyperbolic space, and also a holonomy representation
\[
\mathrm{hol} \colon \pi_1M \hookrightarrow \mathrm{Isom}^+(\Hyp^3)
\]
which is an embedding from the fundamental group of $M$ to the orientation-preserving isometry group of three-dimensional hyperbolic space. For details, see, e.g., \cite{Thurston}, \cite{Ratcliffe} or \cite{Martelli}. Noting that $\mathrm{Isom}^+(\Hyp^3) \cong \mathrm{PSL}(2,\C)$, this gives us an embedding of the fundamental group into a two-dimensional linear group. In the case that $M$ is triangulated, the decomposition of $M$ into simplices lifts to one of $\widetilde M$, which gives rise to a labelling of the $0$-skeleton of the lifted triangulation by elements of $\partial \Hyp^3 = \C P^1$ and this labelling encodes all of the information necessary to construct the holonomy representation; see $\cite{LuoTillmannYang}$ for details. In \cite{Luo}, Luo generalized these labellings to labellings over $\mathrm{P}^1(R)$, the projective line over an arbitrary commutative ring $R$, and constructed representations into $\mathrm{PGL}(2,R)$ by solving  Thurston's hyperbolic gluing equations over $R$. See \cite[Section 3.8]{FriedlGillTillmann} for an example illustrating the strength of Luo’s generalization. Thus, for triangulated connected compact $3$-manifolds $M$, Luo is able to construct an ample supply of two-dimensional linear representations, over commutative rings, of $\pi_1M$. Based on this, Luo makes the following conjecture.

\begin{Conjecture}[\cite{Luo}]
If $M$ is a connected compact $3$–manifold and $\gamma \in \pi_1M \setminus \{1\}$, then there exists a finite commutative ring $R$ and a homomorphism $\pi_1M \to \mathrm{PGL}(2,R)$ whose kernel does not contain $\gamma$.
\end{Conjecture}

It is known that the fundamental groups of compact $3$-manifolds are residually finite; were this conjecture true, it would provide a concrete list of finite groups into which one could map without killing a given element. However, in \cite{FriedlGillTillmann} and \cite{Gill}, we showed that the conjecture is false. In particular, we demonstrated that there exists a closed graph $3$-manifold which does not satisfy the conjecture. On the other hand, the conjecture arose in the context of orientable connected compact geometric $3$-manifolds, and so, in this paper we investigate the conjecture for these $3$-manifolds. We show that the conjecture in fact fails even in this case, but is true for six of the eight Thurston model geometries, namely all but $\Sph^3$ and $\widetilde{\mathrm{SL}_2}$. In the case of $\Sph^3$, the situation is recoverable in the sense that the conjecture holds true if we replace projective general linear groups with general linear groups. On the other hand, in the case of $\widetilde{\mathrm{SL}_2}$, the conjecture does not hold true even in this form. \\

In fact, by Proposition \ref{prop:equivalent_characterizations} below, a finitely generated group satisfies the residual property referred to in the conjecture if and only if it embeds into $\mathrm{PGL}(2,R)$ for some, not necessarily finite, commutative ring $R$. Thus, the main result of the paper is the following.

\begin{Theorem}
If $M$ is an orientable connected compact $3$-manifold which admits a geometric structure modelled on any one of Thurston's eight model geometries, except $\Sph^3$ and $\widetilde{\mathrm{SL}_2}$, then $\pi_1M$ embeds into $\mathrm{PGL}(2,R)$ for some commutative ring $R$; equivalently, given any $\gamma \in \pi_1M \smallsetminus \{1\}$, there exists a homomorphism $\pi_1M \to \mathrm{PGL}(2,R_{\mathrm{fin}})$, for some finite commutative ring $R_{\mathrm{fin}}$, whose kernel does not contain $\gamma$. In the case of $\Sph^3$, the same holds but with $\mathrm{PGL}$ replaced by $\mathrm{GL}$; a counterexample for the projective case is provided by the Poincar\'{e} homology sphere $\Sigma(2,3,5)$. In the case of $\widetilde{\mathrm{SL}_2}$, a counterexample, for both the projective and the non-projective case, is provided by the Brieskorn sphere $\Sigma(2,3,7)$.
\end{Theorem}

\section{Universal Representations and Equivalent Characterizations}

In this section, we recall some notation, terminology and results from our previous works \cite{FriedlGillTillmann} and \cite{Gill} which will be helpful here. In the course of our investigations, the target group of a representation is often any one of the following four: general linear, special linear, projective general linear and projective special linear. One can ask the same question as that in Luo's conjecture with the projective general linear group replaced with any of the other three. To treat all of these cases uniformly, we make the following definition.

\begin{Definition}
\label{def:residual_properties}
Let $\mathrm{K}$ denote any of the symbols $\mathrm{GL}_2$, $\mathrm{SL}_2$, $\mathrm{PGL}_2$ and $\mathrm{PSL}_2$. Given a group $G$, say that $G$ is \emph{residually $\mathrm{K}$} if, for any $g \in G \smallsetminus \{1\}$, there exists a commutative ring $R$ and a homomorphism $G \to \mathrm{K}(R)$ whose kernel does not contain $g$. If it is only known that this property holds at a given fixed element $g_0$, we say that $G$ is \emph{residually $\mathrm{K}$ at $g_0$}. If we can take the commutative ring $R$ to be finite, we say that $G$ is \emph{residually $\mathrm{K}$-finite}, or $G$ is \emph{residually $\mathrm{K}$-finite at $g_0$}, respectively.
\end{Definition}

The following result provides the relations between the residual $\mathrm{K}$-finiteness properties as we vary $\mathrm{K}$. See \cite[Proposition 2.9]{FriedlGillTillmann} and \cite[Section 2]{Gill} for the proof.

\begin{Proposition}
\label{prop:residual_K_finiteness_for_various_K}
Let $G$ be a finitely generated group. We have the following implications for $G$, which hold both with general residual finiteness and with residual finiteness at a specified element of $G$.
\begin{center}
 \begin{tikzpicture}
 \draw[-implies, double equal sign distance] (0.2,2.1) node[anchor=east] {\textit{residually $\mathrm{SL}_2$--finite\,\,}} -- (3,2.1)  node[anchor=west] 
 {\textit{\,\,residually $\mathrm{GL}_2$--finite}};
 \draw[-implies, double equal sign distance] (3,1.9) -- (0.2,1.9) node[anchor=north,midway] {\textit{$G_{\mathrm{ab}}$ t.f.}};

 \draw[-implies, double equal sign distance] (-1.25,1.7) -- (-1.25,0.75) node[anchor=north] {\hspace{-8mm} \textit{residually $\mathrm{PSL}_2$--finite\,\,}} node[anchor=east,midway] {\textit{$Z(G)$ 2-t.f.}};

 \draw[-implies, double equal sign distance] (4.5,1.7) -- (4.5,0.75) node[anchor = north] 
 {\hspace{13mm} \textit{\hspace{-5mm} \,\,residually $\mathrm{PGL}_2$--finite.}} node[anchor=west,midway] {\textit{$\mathrm{Z}(G) = 1$}}; 
 
 \draw[-implies,double equal sign distance] (0.4,0.6) -- (3,0.6) node[anchor=south,midway] {\textit{}};
 \draw[-implies,double equal sign distance] (3,0.4) -- (0.4,0.4) node[anchor=north,midway] {\textit{}};
 \end{tikzpicture}
\end{center}
Here, t.f.\ means torsion-free, 2-t.f.\ means 2-torsion-free, $G_{\mathrm{ab}}$ denotes the abelianization of $G$, and $\mathrm{Z}(G)$ denotes the centre of $G$.
\qed
\end{Proposition}

In the case that $G$ is finitely generated, which of course includes the case of the fundamental groups of compact $3$-manifolds, the following result gives alternative characterizations of the residual $\mathrm{K}$-finiteness properties. See \cite[Proposition 2.2, Corollary 2.3, Corollary 2.5]{FriedlGillTillmann} for the proof.

\begin{Proposition}
\label{prop:equivalent_characterizations}
Let $G$ be a finitely generated group, and let $\mathrm{K}$ denote any of the symbols $\mathrm{GL}_2$, $\mathrm{SL}_2$, $\mathrm{PGL}_2$, $\mathrm{PSL}_2$. We have the following:
\begin{itemize}
	\item If $G$ is residually $\mathrm{K}$ at $g$, then it is residually $\mathrm{K}$-finite at $g$.
	\item The following are equivalent:
\begin{itemize}
	\item $G$ is residually $\mathrm{K}$-finite.
	\item $G$ is residually $\mathrm{K}$.
	\item $G$ admits a faithful representation into $\mathrm{K}(R)$ for some, not necessarily finite, commutative ring $R$.
\end{itemize}
\end{itemize}\qed
\end{Proposition}

Thus we can check residual $\mathrm{K}$-finiteness by looking for faithful representations. The next result provides a universal representation through which all representations factor, and so, in a sense, we need only study this universal representation. See \cite[Proposition 2.6]{FriedlGillTillmann} for the proof and details.

\begin{Proposition}
\label{prop:universal_representation}
Let $G$ be a finitely generated group, and let $\mathrm{K}$ denote any of the symbols $\mathrm{GL}_2$, $\mathrm{SL}_2$, $\mathrm{PGL}_2$, $\mathrm{PSL}_2$. Then there exists a commutative ring $S_{\mathrm{K}}$, an ideal $I_{\mathrm{K}} \trianglelefteq S_{\mathrm{K}}$ and a homomorphism $\varphi_{\mathrm{K}} \colon G \rightarrow \mathrm{K}(S_{\mathrm{K}}/I_{\mathrm{K}})$ such that any representation $G \to \mathrm{K}(R)$ factors through $\varphi_{\mathrm{K}}$; that is, for each $\rho \colon G \to \mathrm{K}(R)$, there exists a map $\psi \colon \mathrm{K}(S_{\mathrm{K}}/I_{\mathrm{K}}) \to \mathrm{K}(R)$ such that the following diagram commutes:
\begin{center}
  \begin{tikzpicture}
  \draw[->] (0.2,2) node[anchor=east] {$G$} -- (2,2)  node[anchor=west] 
  {$\mathrm{K}(S_{\mathrm{K}}/I_{\mathrm{K}})$}
  node[anchor=south,midway] {$\varphi_{\mathrm{K}}$};
  \draw[->] (0.1,1.7) -- (2,0.75) node[anchor=east,midway,yshift=-3mm,xshift=1mm] {$\rho$};
  \draw[->] (2.7,1.7) -- (2.7,0.75) node[anchor = north] 
  {$\mathrm{K}(R).$} node[anchor=west,midway] {$\psi$};
  \end{tikzpicture}
\end{center}
\qed
\end{Proposition}

In the case of finitely generated abelian groups, the residual $\mathrm{K}$-finiteness properties are always satisfied. See Proposition \cite[Proposition 3.5]{FriedlGillTillmann}.

\begin{Proposition}
\label{prop:residual_finiteness_for_abelian_groups}
If $G$ is a finitely generated abelian group, then $G$ is residually $\mathrm{K}$-finite for $\mathrm{K}$ any of $\mathrm{GL}_2$, $\mathrm{SL}_2$, $\mathrm{PGL}_2$, $\mathrm{PSL}_2$.
\qed
\end{Proposition}

The following result shows that residual $\mathrm{GL}_2$, $\mathrm{SL}_2$, $\mathrm{PGL}_2$, $\mathrm{PSL}_2$-finiteness are preserved by products.

\begin{Proposition}
\label{prop:products_preserve_residual_k_finiteness}
Let $\{G_i\}_{i \in I}$ be a collection of groups, for any index set $I$, and let $\mathrm{K}$ denote any of the symbols $\mathrm{GL}_2$, $\mathrm{SL}_2$, $\mathrm{PGL}_2$, $\mathrm{PSL}_2$. If, for each $i \in I$, $G_i$ is residually $\mathrm{K}$-finite, so is the product $\prod G_i$. 
\end{Proposition}

\begin{proof}
Given any non-identity element $(g_i)$ of $\prod G_i$, there must be some $i_0 \in I$ for which $g_{i_0} \neq 1$. Given that $G_{i_0}$ is residually $\mathrm{K}$-finite, there exists a homomorphism $G_{i_0} \to \mathrm{K}(R)$ for some finite commutative ring $R$ which does not kill $g_{i_0}$. The composite homomorphism $\prod G_i \to G_{i_0} \to \mathrm{K}(R)$ then does not kill $(g_i)$.
\end{proof}

\section{The Centres of Perfect Subgroups of $\mathrm{GL}(2,R)$ and $\mathrm{PGL}(2,R)$}

In this section, we investigate the centres of perfect subgroups of both general linear groups and projective general linear groups. The result at the end of the section, Proposition \ref{prop:perfect_group_counterexample}, will be used later for counterexamples. First, we need a lemma.

\begin{Lemma}
\label{lem:span_of_perfect_subgroups_of_GL_2}
If $k$ is a field and $G$ a non-trivial perfect subgroup of $\mathrm{GL}(2,k)$, then $\mathrm{span}_k(G) = \mathrm{M}(2, k)$.
\end{Lemma}

Here, $\mathrm{span}_k(G)$ denotes the $k$-linear span of the matrices of $G$, and $\mathrm{M}(2, k)$ denotes the $k$-algebra of all $2 \times 2$ matrices over $k$.

\begin{proof}
Let $A = \mathrm{span}_k(G)$. Let $I_2$ denote the $2 \times 2$ identity matrix. As $I_2 \in G$, we have that $I_2 \in A$. Moreover, as $G$ is closed under multiplication, so is $A$. Thus $A$ is $k$-subalgebra of $\mathrm{M}(2, k)$. As $\mathrm{M}(2, k)$ is four-dimensional, it follows that $\dim A \in \{0, 1, 2, 3, 4\}$. We will rule out the possibilities $0, 1, 2, 3$ from which it will follow that $\dim A = 4$, and so that $A = \mathrm{M}(2, k)$, as desired. \\

Case 1: $\dim A = 0$. If this is so, then $G = \{I_2\}$, contradicting the non-triviality of $G$. \\

Case 2: $\dim A = 1$. If this is so, then, since $I_2 \in A$, it follows that $A = kI_2$. Thus, every element of $G$ is a scalar matrix, and so is $G$ is abelian. But then, as $G$ is perfect, we have that $G = [G, G] = 1$, contradicting the non-triviality of $G$. \\ 

Case 3: $\dim A = 2$. If this is so, since $I_2 \in A$, there exists some $X \in A$ such that $A$ is the $k$-linear span of $I_2$ and $X$. In particular then, every element of $G$ has the form $aI_2 + bX$. Noting that $(aI_2 + bX)(cI_2 + dX) = (cI_2 + dX)(aI_2 + bX) = acI_2 + (ad+bc)X + bdX^2$, it follows that $G$ is abelian. As in Case 2, using the perfectness of $G$, we find that $G = [G, G] = 1$, contradicting the non-triviality of $G$. \\

Case 4: $\dim A = 3$. If this is so, we claim that there exists a one-dimensional subspace of $k^2$ which is preserved by all of $A$. To see this, first, let $E_1, E_2, E_3$ constitute a basis for $A$. For an arbitrary matrix $X$, note that the conditions
\[
\mathrm{tr}(XE_1) = 0
\hspace{1cm}
\mathrm{tr}(XE_2) = 0
\hspace{1cm}
\mathrm{tr}(XE_3) = 0
\]
form a system of three homogeneous linear equations in the four entries of $X$. Moreover, linear independence of the $E_i$ implies linear independence of the conditions. It follows that the system above has a one-dimensional solution space. Let $T$ be a non-zero such solution. All other solutions are then multiples of $T$. Note that, as the $E_i$ form a basis for $A$, we have that $\mathrm{tr}(TX) = 0$ for all $X \in A$. \\ 

Now, fix some $X \in A$, and consider arbitrary $Y \in A$. We have that $\mathrm{tr}((XT)Y) = \mathrm{tr}(T(YX)) = 0$, where the second equality follows because $YX \in A$. Thus $XT$ is itself a solution to the linear system above, and so it follows that $XT$ is a multiple of $T$. Thus, for all $X \in A$, we have that there exists some $\lambda_X \in k$ such that $XT = \lambda_X T$. Note that then $T$ cannot be invertible, as otherwise we would find that each $X$ is a scalar matrix, contradicting $\dim A = 3$. \\

Now, we have established that $T$ is a non-zero, singular $2 \times 2$ matrix. It follows that $T$ has rank $1$, and that it can be written in the form $T = uv^{\mathsf T}$ for some non-zero column vectors $u, v \in k^2$. For any matrix $X$, a direct calculation shows that $\mathrm{tr}(uv^{\mathsf T}X) = v^{\mathsf T}Xu$. Thus, for all $X \in A$, we have that $v^{\mathsf T}Xu = 0$. Putting $X = I_2$, we have that $v^{\mathsf T}u = 0$; that is, $u$ and $v$ are orthogonal. It follows that $\ker(v^{\mathsf T} \cdot (-)) = \langle u \rangle$. Thus, given any $X \in A$, from $v^{\mathsf T}Xu = 0$, we find that $Xu \in \langle u \rangle$. The desired one-dimensional $A$-invariant subspace is then precisely $\langle u \rangle$. \\

Next, choose a basis $(u,w)$ of $k^2$, with $u$ as the first basis vector. Relative to this basis, every matrix $X \in G$ has the form
\[
\left(
\begin{array}{cc}
a & b \\
0 & d
\end{array}
\right)
\]
where $a, d \neq 0$. Thus $G$ is conjugate to a subgroup of the group of invertible upper triangular matrices. The group of invertible upper triangular matrices is known to be solvable, and so it follows that $G$ is solvable. As $G$ is perfect, its derived series is constant at $G$, and so $G = 1$, contradicting the non-triviality of $G$.
\end{proof}

\begin{Proposition}
\label{prop:finite_local_ring_centre_of_SL_2}
If $R$ is a finite commutative ring, then the centre of any perfect subgroup of $\mathrm{GL}(2,R)$ must lie among the scalar matrices.
\end{Proposition}

\begin{proof}
First, we reduce to the case where $R$ is a finite local ring. Because $R$ is a finite commutative ring, by \cite[Theorem 8.7]{AtiyahMacDonald}, it uniquely decomposes as a direct product of finite local rings: $R \cong \prod_{i=1}^n R_i$. This induces a canonical isomorphism $\mathrm{M}(2, R) \cong \prod_{i=1}^n \mathrm{M}(2, R_i)$, which restricts to $\mathrm{GL}(2, R) \cong \prod_{i=1}^n \mathrm{GL}(2, R_i)$. Let $\Gamma$ be a perfect subgroup of $\mathrm{GL}(2, R)$. Let also $\Gamma_i$ be the projection of $\Gamma$ into $\mathrm{GL}(2, R_i)$. Since $\Gamma$ is perfect, each $\Gamma_i$ is a perfect subgroup of $\mathrm{GL}(2, R_i)$. Moreover, if $\gamma \in \mathrm{Z}(\Gamma)$, then, for each $i$, the projection $\gamma_i$ lies in $\mathrm{Z}(\Gamma_i)$. It follows that it suffices to demonstrate the result in the case where $R$ is local. \\

Let $H \le \mathrm{GL}(2, R)$ be a perfect subgroup. Let $\mathfrak{m}$ denote the maximal ideal of $R$, and $k$ the finite residue field $R/\mathfrak{m}$. Consider the reduction map $\rho \colon \mathrm{GL}(2, R) \to \mathrm{GL}(2, k)$. The image $\bar{H} = \rho(H)$ is a perfect subgroup of $\mathrm{GL}(2, k)$. There are now two cases. \\

Case 1: $\bar{H}$ is trivial. Then $H$ is contained entirely within the congruence subgroup $C = I + \mathrm{M}(2, \mathfrak{m})$. Because $\mathfrak{m}$ is the maximal ideal of a finite local ring, it is nilpotent (this follows, e.g., from \cite[Proposition 8.6]{AtiyahMacDonald}). Let $N > 0$ be such that $\mathfrak{m}^N = 0$ and consider the terminating descending chain
\begin{equation}
\label{eqn:finite_local_ring_lemma_nilpotent_filtration}
C_1 \supseteq C_2 \supseteq \cdots \supseteq C_N = \{I\}
\end{equation}
where $C_i = I + \mathrm{M}(2,\mathfrak{m}^i)$. Note that $C_1 = C$. If $g \in C_1$ and $h \in C_i$, a simple check shows that, modulo $\mathrm{M}(2,\mathfrak{m}^{i+1})$, $gh$ and $hg$ coincide. It then follows that $[g,h] \in C_{i+1}$. Thus, for each $i$, $[C_1,C_i] \subseteq C_{i+1}$. Now, we have, as above, that $H \subseteq C_1$. As $H$ is perfect, we have that $H = [H, H] \subseteq [C_1, C_1] \subseteq C_2$. From this, it then similarly follows that $H = [H, H] \subseteq [C_1, C_2] \subseteq C_3$. Via induction, we find that $H \subseteq C_i$ for all $i$, and so, as the $C_i$ terminate, we have that $H$ is trivial. The desired conclusion then is trivially satisfied. \\

Case 2: $\bar{H}$ is non-trivial. By Lemma \ref{lem:span_of_perfect_subgroups_of_GL_2}, we have that the $k$-linear span of $\bar{H}$ is the entire matrix algebra $\mathrm{M}(2, k)$. We claim that then the $R$-linear span of $H$ is the entire matrix algebra $\mathrm{M}(2, R)$. To see this, let $h_1, h_2, h_3, h_4 \in H$ be such that $\bar{h}_1, \bar{h}_2, \bar{h}_3, \bar{h}_4$ form a basis of $\mathrm{M}(2, k)$. Using the standard basis of $\mathrm{M}(2, R)$, write each $h_i$ as a column vector. Thus, given a matrix
\[
X =
\left(
\begin{array}{cc}
a & b \\
c & d
\end{array}
\right),
\]
we identify it with the column vector
\[
\mathrm{vec}(X) =
\left(
\begin{array}{cc}
a \\
b \\
c \\
d
\end{array}
\right).
\]
Now form the $4 \times 4$ matrix
\[
M =
\left(
\begin{array}{cccc}
 | & | & | & | \\
\mathrm{vec}(h_1) &
\mathrm{vec}(h_2) &
\mathrm{vec}(h_3) &
\mathrm{vec}(h_4) \\
 | & | & | & |
\end{array}
\right).
\]
Reducing its entries modulo \(\mathfrak m\), we get
\[
\bar{M} =
\left(
\begin{array}{cccc}
 | & | & | & | \\
\mathrm{vec}(\bar{h}_1) &
\mathrm{vec}(\bar{h}_2) &
\mathrm{vec}(\bar{h}_3) &
\mathrm{vec}(\bar{h}_4) \\
 | & | & | & |
\end{array}
\right).
\]
Because the $\bar{h}_i$ form a basis of $\mathrm{M}(2,k)$, the columns of $\bar{M}$ are linearly independent, and so $\det(\bar{M}) \neq 0$ in $k$. Determinants commute with reduction modulo $\mathfrak{m}$, and so we have that $\det(\bar{M})=\overline{\det M}$. Hence $\overline{\det M} \neq 0$. This means that $\det M \notin \mathfrak{m}$. Since $R$ is local, every element outside $\mathfrak{m}$ is invertible, and so $\det M \in R^\times$. It follows that $M$ is invertible, and so that the $h_i$ span $\mathrm{M}(2, R)$, as desired. \\

Now, as the $R$-linear span of $H$ is all of $\mathrm{M}(2, R)$, any element in $\mathrm{Z}(H)$ must commute with the entire matrix ring $\mathrm{M}(2, R)$. The centre of $\mathrm{M}(2, R)$ is exactly the set of scalar matrices, and so the desired conclusion follows.
\end{proof}

\begin{Proposition}
\label{prop:finite_ring_centre_of_perfect_subgroup_of_PSL_2}
If $R$ is a finite commutative ring, then the centre of any perfect subgroup of $\mathrm{PGL}(2, R)$ must be trivial.
\end{Proposition}

\begin{proof}
Let $H$ be a perfect subgroup of $\mathrm{PGL}(2, R)$. Let $\pi \colon \mathrm{GL}(2, R) \to \mathrm{PGL}(2, R)$ be the canonical projection and let $\mathrm{U}(2,R)$ denote the kernel of $\pi$, that is, the group of scalar matrices $R^\times I_2$. Set $\tilde{H} = \pi^{-1}(H)$ and $\Gamma = [\tilde{H}, \tilde{H}]$. We claim that $\Gamma$ is a perfect subgroup of $\mathrm{SL}(2,R)$. To see this, first note that, as $\Gamma$ is generated by commutators, all constituent matrices have unit determinant. Next, note that:
\[
\pi(\Gamma) = \pi([\tilde{H}, \tilde{H}]) = [\pi(\tilde{H}), \pi(\tilde{H})] = [H, H] = H.
\]
Now, given $x, y \in \tilde{H}$, we need to show that $[x,y] \in [\Gamma, \Gamma]$. As $\pi(x) \in H$, and $\pi(\Gamma) = H$, we have that $\pi(x) = \pi(\gamma)$ for some $\gamma \in \Gamma$, and so $x = \gamma u$ for some $\gamma \in \Gamma$ and $u \in \mathrm{U}(2,R)$. Similarly, $y = \gamma' u'$ for some $\gamma' \in \Gamma$ and $u' \in \mathrm{U}(2,R)$. We then have that
\[
[x,y] = (\gamma u)(\gamma' u')(\gamma u)^{-1}(\gamma' u')^{-1}.
\]
As $u$ and $u'$ are scalar, and so central in $\mathrm{GL}(2,R)$, they and their inverses can be eliminated, and so we are left with
\[
[x,y] = [\gamma,\gamma'] \in [\Gamma, \Gamma]
\]
as desired. We have now demonstrated that $\Gamma$ is a perfect subgroup of $\mathrm{SL}(2, R)$. \\

Fix $w \in \mathrm{Z}(H)$. Our goal is to show that $w = 1$. Because $\pi(\Gamma) = H$, there exists some $\gamma \in \Gamma$ such that $\pi(\gamma) = w$. Fix such a preimage, say $\gamma_0$. We claim that $\gamma_0 \in \mathrm{Z}(\Gamma)$. To see this, first, note that, for any $\gamma \in \Gamma$, because $w$ is central in $H$, we have:
\[
\pi([\gamma_0, \gamma]) = [\pi(\gamma_0), \pi(\gamma)] = [w, \pi(\gamma)] = 1.
\]
This implies that $[\gamma_0, \gamma] \in \mathrm{U}(2,R)$. In other words, for every $\gamma \in \Gamma$, the commutator $[\gamma_0, \gamma]$ is a central scalar matrix in $\mathrm{GL}(2, R)$. We thus find that $\gamma \mapsto [\gamma_0, \gamma]$ defines a map $c_{\gamma_0} \colon \Gamma \to \mathrm{U}(2,R)$. This is in fact a homomorphism. To see this, note that, for any $\gamma, \delta \in \Gamma$, we have
\[
c_{\gamma_0}(\gamma\delta) = \gamma_0\gamma\delta\gamma_0^{-1}\delta^{-1}\gamma^{-1} = (\gamma_0\gamma\gamma_0^{-1}\gamma^{-1})\gamma(\gamma_0\delta\gamma_0^{-1}\delta^{-1})\gamma^{-1} = [\gamma_0,\gamma]\gamma[\gamma_0, \delta]\gamma^{-1}.
\]
Because $[\gamma_0,\gamma], [\gamma_0, \delta]$ are scalar, they commute with $\gamma$ and $\gamma^{-1}$, allowing us to write
\[
c_{\gamma_0}(\gamma\delta) = [\gamma_0, \gamma][\gamma_0, \delta] = c_{\gamma_0}(\gamma)c_{\gamma_0}(\delta).
\]
As $\Gamma$ is perfect, and $\mathrm{U}(2,R)$ abelian, we must have that $c_{\gamma_0}$ is trivial. Thus, $[\gamma_0, \gamma] = 1$ for all $\gamma \in \Gamma$, and so we find that $\gamma_0 \in \mathrm{Z}(\Gamma)$, as desired. By Lemma \ref{prop:finite_local_ring_centre_of_SL_2}, we have that $\gamma_0$ is scalar, and so $w = \pi(\gamma) = 1$, as desired.
\end{proof}

We can now state the result which will be used later in relation to counterexamples.

\begin{Proposition}
\label{prop:perfect_group_counterexample}
Let $R$ be a finite commutative ring, $G$ a perfect group and $z \in \mathrm{Z}(G)$. We have the following:
\begin{itemize}
	\item[(i)] For any homomorphism $\varphi \colon G \to \mathrm{GL}(2, R)$, we have $\varphi(z) = \mu I$ with $\mu^2 = 1$; in particular, $\varphi(z^2) = 1$.
	\item[(ii)] For any homomorphism $\varphi \colon G \to \mathrm{PGL}(2, R)$, we have $\varphi(z) = 1$.
\end{itemize}
\end{Proposition}

\begin{proof}
(i): Let $H = \varphi(G) \subseteq \mathrm{GL}(2, R)$. Since $G$ is perfect, $H$ is a perfect subgroup of $\mathrm{GL}(2, R)$.  Because $z \in \mathrm{Z}(G)$, we have that $\varphi(z) \in \mathrm{Z}(H)$. It then follows by Lemma \ref{prop:finite_local_ring_centre_of_SL_2} that $\varphi(z)$ is a scalar. As $G$ is perfect, $\det \circ \, \varphi$ must be trivial, and so $\varphi(z)$ must have unit determinant. The result follows. \\

(ii): Let $H = \varphi(G) \subseteq \mathrm{PGL}(2, R)$. Since $G$ is perfect, $H$ is a perfect subgroup of $\mathrm{PGL}(2, R)$. Because $z \in \mathrm{Z}(G)$, we have that $\varphi(z) \in \mathrm{Z}(H)$. It then follows by Lemma \ref{prop:finite_ring_centre_of_perfect_subgroup_of_PSL_2} that $\varphi(z) = 1$.
\end{proof}

\section{Hyperbolic Geometry $\Hyp^3$ and the Trivial Product Geometries $\Hyp^2 \times \E$, $\Sph^2 \times \E$}

In our previous work \cite[Proposition 3.7]{FriedlGillTillmann}, we claimed to have demonstrated residual $\mathrm{K}$-finiteness, for each of $\mathrm{K} = \mathrm{GL}_2, \mathrm{SL}_2, \mathrm{PGL}_2, \mathrm{PSL}_2$, for the cases of the geometries $\Hyp^3$, $\Hyp^2 \times \E$ and $\Sph^2 \times \E$. The proofs carried two flaws: (i) they demonstrated residual $\mathrm{PGL}_2, \mathrm{PSL}_2$-finiteness and concluded $\mathrm{GL}_2, \mathrm{SL}_2$-finiteness via the erroneous (see \cite{Gill}) Proposition 2.9 (ii) $\mathrm{Isom}^+(\Hyp^2 \times \E)$ is not exactly $\mathrm{Isom}^+(\Hyp^2) \times \mathrm{Isom}^+(\E)$ and $\mathrm{Isom}^+(\Sph^2 \times \E)$ is not exactly $\mathrm{Isom}^+(\Sph^2) \times \mathrm{Isom}^+(\E)$, as isometries which are orientation-reversing on both factors yield orientation-preserving isometries on the products. We provide corrected arguments and conclusions here.

\begin{Theorem}
\label{thm:hyperbolic_geometry}
If $M$ is an orientable connected compact 3-manifold and admits a geometric structure modelled on $\Hyp^3$, then $\pi_1(M)$ is residually $\mathrm{GL}_2, \mathrm{SL}_2, \mathrm{PGL}_2, \mathrm{PSL}_2$-finite.
\end{Theorem}

\begin{proof}
Using holonomy representations, residual $\mathrm{PGL}_2, \mathrm{PSL}_2$-finiteness follow from $\mathrm{Isom}^+(\Hyp^3) \cong \mathrm{PSL}(2,\C)$ and Propositions \ref{prop:equivalent_characterizations} and \ref{prop:residual_K_finiteness_for_various_K}. As for residual $\mathrm{GL}_2, \mathrm{SL}_2$-finiteness, these follow from a well-known result due to Thurston that, for a given hyperbolic 3-manifold, we can lift the holonomy representation into $\mathrm{PSL}(2, \C)$ to one into $\mathrm{SL}(2, \C)$. (For the non-projective cases, this same reasoning is in fact provided also in \cite{FriedlGillTillmann} as an alternative argument, in place of the invocation of the erroneous Proposition 2.9.)
\end{proof}

\begin{Theorem}
\label{thm:trivial_product_geometries}
If $M$ is an orientable connected compact 3-manifold and admits a geometric structure modelled on $\Hyp^2 \times \E$ or $\Sph^2 \times \E$, then $\pi_1(M)$ is residually $\mathrm{PGL}_2, \mathrm{PSL}_2$-finite.
\end{Theorem}

\begin{proof}
For $\Hyp^2 \times \E$, first, note that $\mathrm{Isom}^+(\Hyp^2 \times \E) \subseteq \mathrm{Isom}(\Hyp^2) \times \mathrm{Isom}(\E)$. We have that $\mathrm{Isom}(\Hyp^2) \cong \mathrm{PGL}(2,\R)$. Moreover, by encoding affine maps as upper triangular matrices, we have that $\mathrm{Isom}(\E)$ embeds into $\mathrm{PGL}(2,\R)$. Thus $\mathrm{Isom}^+(\Hyp^2 \times \E) \hookrightarrow \mathrm{PGL}(2,\R) \times \mathrm{PGL}(2,\R) \cong \mathrm{PGL}(2, \R \times \R)$. The result now follows via holonomy representations and Propositions \ref{prop:equivalent_characterizations} and \ref{prop:residual_K_finiteness_for_various_K}. For $\Sph^2 \times \E$, first, note that $\mathrm{Isom}^+(\Sph^2 \times \E) \subseteq \mathrm{Isom}(\Sph^2) \times \mathrm{Isom}(\E)$. We have that $\mathrm{Isom}(\Sph^2) \cong \mathrm{O}(3, \R)$ and $\mathrm{O}(3, \R) \cong \mathrm{SO}(3, \R) \times \Z_2$, where $\Z_2$ is the cyclic group of order 2, and where the latter isomorphism is given by $A \mapsto ((\det A)A, \det A)$. Now, we have that $\mathrm{SO}(3,\R) \cong \mathrm{PSU}(2,\C) \subseteq \mathrm{PGL}(2, \C)$, and $\Z_2$ of course embeds into $\mathrm{PGL}(2, \C)$. As before, $\mathrm{Isom}(\E)$ embeds into $\mathrm{PGL}(2,\R)$. Thus $\mathrm{Isom}^+(\Sph^2 \times \E) \hookrightarrow \mathrm{PGL}(2,\C) \times \mathrm{PGL}(2,\C) \times \mathrm{PGL}(2, \R) \cong \mathrm{PGL}(2, \C \times \C \times \R)$. The result now follows via holonomy representations and Propositions \ref{prop:equivalent_characterizations} and \ref{prop:residual_K_finiteness_for_various_K}.
\end{proof}

\section{Euclidean Geometry $\mathbb{E}^3$}

Let $M$ be an orientable connected compact 3-manifold with a geometric structure modelled on $\E^3$. From the geometric structure, we have a faithful holonomy representation $\mathrm{hol} \colon \pi_1M \hookrightarrow \mathrm{Isom}^+(\E^3)$. It is standard that $\mathrm{Isom}^+(\E^3) \cong \R^3 \rtimes \mathrm{SO}(3, \R)$ where multiplication in the latter is given by
\[
(v, R) \cdot (w, S) = (v + Rw, RS).
\]
We shall construct a linear embedding of $\mathrm{Isom}^+(\E^3)$ by first constructing a 2-dimensional embedding of $\R^3$. \\

Let $\varepsilon$ be an indeterminate, and consider the ring $\C[\varepsilon]/(\varepsilon^2)$. For $v = (x,y,z) \in \R^3$, set
\[
\mathrm{X}(v) =
\left(
\begin{array}{cc}
ix & y+iz \\
-y+iz & -ix
\end{array}
\right).
\]
Note that, for $v, w \in \R^3$, we have
\begin{equation}
\label{eqn:euc_proof_X_additive}
\mathrm{X}(v+w) = \mathrm{X}(v) + \mathrm{X}(w).
\end{equation}
In fact, $\mathrm{X}$ identifies $\R^3$ with the Lie algebra $\mathfrak{su}(2)$, where the standard basis vectors $e_1$, $e_2$ and $e_3$ map to $i\sigma_3$, $i\sigma_2$, $i\sigma_1$, where $\sigma_1$, $\sigma_2$ and $\sigma_3$ are the Pauli matrices, namely
\[
\sigma_1 = 
\left(
\begin{array}{cc}
0 & 1 \\
1 & 0
\end{array}
\right)
,\quad
\sigma_2 = 
\left(
\begin{array}{cc}
0 & -i \\
i & 0
\end{array}
\right)
,\quad
\sigma_3 = 
\left(
\begin{array}{cc}
1 & 0 \\
0 & -1
\end{array}
\right).
\]
Now define
\[
\mathrm{T}(v) := I + \varepsilon\mathrm{X}(v).
\]
Here, the matrix entries lie in $\C[\varepsilon]/(\varepsilon^2)$. Note that, since $\mathrm{X}(v)$ is traceless, $\det(I + \varepsilon\mathrm{X}(v)) = 1 + \varepsilon\,\mathrm{tr}\,\mathrm{X}(v) + \varepsilon^2\det\mathrm{X}(v) = 1$. We claim that $\mathrm{T}$ is a group embedding of $\R^3$ in $\mathrm{SL}(2, \C[\varepsilon]/(\varepsilon^2))$. To see this, note that, using (\ref{eqn:euc_proof_X_additive}) and $\varepsilon^2 = 0$, we have
\begin{align*}
(I + \varepsilon\mathrm{X}(v)) \cdot (I + \varepsilon\mathrm{X}(w)) &= I + \varepsilon(\mathrm{X}(v) + \mathrm{X}(w)) \\
&= I + \varepsilon\mathrm{X}(v+w).
\end{align*}
Moreover, a straightforward entry-by-entry check shows that $\mathrm{T}$ is faithful. \\

In fact, if we projectivize, that is to say postcompose with the projection $\mathrm{SL}_2(\C[\varepsilon]/(\varepsilon^2)) \to \mathrm{PSL}_2(\C[\varepsilon]/(\varepsilon^2))$ to form a representation
\[
\widehat{\mathrm{T}} \colon \R^3 \longrightarrow \mathrm{PSL}_2(\C[\varepsilon]/(\varepsilon^2))
\]
then it too is faithful. To see this, note that if $\mathrm{T}(x,y,z)$ is a scalar, then from the off-diagonal entries, we have $y = z = 0$, and equality of the diagonal entries enforces $2ix\varepsilon = 0$ which forces $x = 0$. \\

We have now constructed a 2-dimensional projective linear embedding of $\R^3$. Our goal is to construct such an embedding for $\mathrm{Isom}^+(\E^3) \cong \R^3 \rtimes \mathrm{SO}(3, \R)$. To do so, we need to address the $\mathrm{SO}(3, \R)$ factor. It is well-known that there is a double cover
\[
\Gamma \colon \mathrm{SU}(2) \to \mathrm{SO}(3)
\]
which satisfies the following: for $U \in \mathrm{SU}(2)$ and $v \in \R^3$, we have
\begin{equation}
\label{eqn:euc_proof_u_conj}
U\mathrm{X}(v)U^{-1} = \mathrm{X}(\Gamma(U)v).
\end{equation}
For each $A \in \mathrm{SO}(3)$, note that $\Gamma^{-1}(A)$ is well-defined up to a sign. If, for each $A$, we also embed the matrix entries of $\Gamma^{-1}(A)$ in $\C[\varepsilon]/(\varepsilon^2)$, we get a map
\[
\Delta \colon \mathrm{SO}(3) \to \mathrm{PSL}_2(\C[\varepsilon]/(\varepsilon^2)).
\]
A straightforward check shows that $\Delta$ is a homomorphism. Now define a representation
\begin{equation}
\label{eqn:xi_rep_of_isom_euc}
\Xi \colon \R^3 \rtimes \mathrm{SO}(3) \longrightarrow \mathrm{PSL}_2(\C[\varepsilon]/(\varepsilon^2))
\end{equation}
by setting
\[
\Xi(v,A) := \widehat{\mathrm{T}}(v)\Delta(A).
\]
By the universal property of semidirect products, we know that this is a well-defined group map so long as, for any $v \in \R^3$ and $A \in \mathrm{SO}(3)$, we have that
\[
\Delta(A)\widehat{\mathrm{T}}(v)\Delta(A)^{-1} = \widehat{\mathrm{T}}(Av).
\]
This holds as follows. Let $U$ denote a preimage under $\Gamma$ of $A$. Using (\ref{eqn:euc_proof_u_conj}), we then have
\begin{align*}
\Delta(A)\widehat{\mathrm{T}}(v)\Delta(A)^{-1} &= [U] \cdot [I+\varepsilon\mathrm{X}(v)] \cdot [U^{-1}] \\
&= [I + \varepsilon U\mathrm{X}(v)U^{-1}] \\
&= [I + \varepsilon X(\Gamma(U)v)] \\
&= [I + \varepsilon X(Av)] \\
&= \widehat{\mathrm{T}}(Av)
\end{align*}
as desired. \\

Finally, we must check that $\Xi$ is faithful. Suppose that $\Xi(v, A) = 1$. If we reduce the matrices modulo $\varepsilon$, $\widehat{\mathrm{T}}(v)$ maps to $1$, whereas $\Delta(A)$ remains unchanged. We thus have $A = \Gamma(I) = I$. We are then left with $\widehat{\mathrm{T}}(v) = 1$, and so, because $\widehat{\mathrm{T}}$ is faithful, we have $v = 0$. This demonstrates that $\Xi$ is faithful.

\begin{Theorem}
\label{theorem:nil_proof}
If $M$ is an orientable connected compact $3$-manifold which admits a geometric structure modelled on $\E^3$, then $\pi_1M$ is residually $\mathrm{PGL}_2$, $\mathrm{PSL}_2$-finite.
\end{Theorem}

\begin{proof}
Given any such manifold $M$, $\pi_1M$ embeds into $\mathrm{Isom}^+(\E^3)$. As in (\ref{eqn:xi_rep_of_isom_euc}) above, this isometry group embeds into $\mathrm{PSL}_2(\C[\varepsilon]/(\varepsilon^2))$. The conclusion now follows from Propositions \ref{prop:residual_K_finiteness_for_various_K} and \ref{prop:equivalent_characterizations}.
\end{proof}

\section{The Geometry $\mathrm{Nil}$}

The fundamental groups of orientable connected compact $3$-manifolds which admit a geometric structure modelled on $\mathrm{Nil}$ are subgroups of the isometry group $\mathrm{Isom}(\mathrm{Nil})$. Here, $\mathrm{Nil}$ is the $3$-dimensional Lie group of upper unitriangular matrices:
\[
\mathrm{Nil} =
\left\{
\left.
\left(
\begin{array}{ccc}
1 & x & z \\
0 & 1 & y \\
0 & 0 & 1
\end{array}
\right)
\,
\right\vert
\,
x, y, z \in \R
\right\}.
\]
As in \cite[p. 468]{Scott}, writing $(x,y,z)$ for the corresponding matrix above, the product is given by
\[
(x,y,z)  \cdot (x', y', z') = (x + x', y + y', z + z' + xy').
\]
For rotational symmetry, it is convenient to change the third coordinate to $t = z - \frac{1}{2}xy$. With this change of coordinates, the product becomes
\[
(x,y,t)  \cdot (x', y', t') = \left(x + x', y + y', t + t' + \frac{1}{2}(xy'-yx')\right).
\]
It is known (see, e.g., \cite{Scott}), that $\mathrm{Isom}(\mathrm{Nil})$ is isomorphic to the semi-direct product $\mathrm{Nil} \rtimes \mathrm{O}(2)$, where, given a matrix $A \in \mathrm{O}(2)$, it acts on $\mathrm{Nil}$ by sending $(x,y,t)$ to $(A(x,y)^{\mathrm{T}}, (\det A)\cdot t)$ (that is, a rotation acts on the horizontal plane and leaves $t$ unchanged, whereas a reflection acts on the horizontal plane and reverses $t$). We shall construct a linear embedding of $\mathrm{Isom}(\mathrm{Nil})$ by first constructing an embedding of $\mathrm{Nil}$. \\

Let $\varepsilon$ be an indeterminate, and consider the ring $\C[\varepsilon]/(\varepsilon^3)$. Set
\[
P = \left( \begin{array}{cc}
0 & 1 \\
1 & 0
\end{array} \right)
,\quad
Q = \left( \begin{array}{cc}
0 & i \\
-i & 0
\end{array} \right)
,\quad
R = \left( \begin{array}{cc}
-2i & 0 \\
0 & 2i
\end{array} \right).
\]
Here, the matrix entries lie in $\C[\varepsilon]/(\varepsilon^3)$. Note that
\begin{equation}
\label{eqn:nil_proof_P_Q_R_matrices}
P^2 = Q^2 = I
,\quad
PQ + QP = 0
,\quad
PQ - QP = R.
\end{equation}
For each $(x,y,t) \in \mathrm{Nil}$, set
\[
\mathrm{N}(x,y,t) := \varepsilon(xP + yQ) + \varepsilon^2tR.
\]
Note that, for each $(x,y,t)$, $\mathrm{N}(x,y,t)$ is divisible by $\varepsilon$, and so $\mathrm{N}(x,y,t)^3 = 0$. As such, for all $(x,y,t)$, the exponential of $N(x,y,t)$ is well-defined. Define
\[
\Phi(x,y,t) := \exp\mathrm{N}(x,y,t).
\]
With the help of (\ref{eqn:nil_proof_P_Q_R_matrices}), one can check that $\Phi(x,y,t) = I + \varepsilon(xP+yQ) + \varepsilon^2(tR+\frac{1}{2}(x^2+y^2)I)$. More concretely, we find that
\begin{equation}
\label{eqn:nil_proof_Phi_x_y_t_explicit}
\Phi(x,y,t) =
\left(
\begin{array}{cc}
1 + \left(\frac{x^2+y^2}{2}-2it\right)\varepsilon^2 & (x+iy)\varepsilon \\
(x-iy)\varepsilon & 1 + \left(\frac{x^2+y^2}{2}+2it\right)\varepsilon^2
\end{array}
\right).
\end{equation}
We claim that $\Phi$ is a group embedding of $\mathrm{Nil}$ into $\mathrm{SL}_2(\C[\varepsilon]/(\varepsilon^3))$. To see this, first note that, because each $\mathrm{N}(x,y,t)$ is traceless, each $\Phi(x,y,t)$ has unit determinant. Next, in order to show that $\Phi$ respects the multiplication of $\mathrm{Nil}$, we first compute $\Phi((x,y,t) \cdot (x',y',t')) = \Phi\left(x + x', y + y', t + t' + \frac{1}{2}(xy'-yx')\right)$ and find that it amounts to
\[
I + \varepsilon((x+x')P+(y+y')Q) + \varepsilon^2\left((t+t')R + \frac{1}{2}(xy'-yx')R + \frac{1}{2}((x+x')^2+(y+y')^2)I\right).
\]
On the other hand, we must compute $\Phi(x,y,t)\Phi(x',y',t')$. To do so, we must evaluate
\[
\left(I + \varepsilon(xP+yQ) + \varepsilon^2\left(tR+\frac{1}{2}(x^2+y^2)I\right)\right) \cdot \left(I + \varepsilon(x'P+y'Q) + \varepsilon^2\left(t'R+\frac{1}{2}((x')^2+(y')^2)I\right)\right).
\]
Upon expansion, we see that the constant term is $I$ and that the coefficient of $\varepsilon$ is $(x+x')P + (y+y')Q$, as desired. As for the coefficient of $\varepsilon^2$, we find that it amounts to
\[
(t+t')R + \frac{1}{2}(x^2+y^2)I + \frac{1}{2}((x')^2 + (y')^2)I + (xx'+yy')I + xy'PQ + yx'QP.
\]
It remains to demonstrate that $xy'PQ + yx'QP = \frac{1}{2}(xy'-yx')R$. Using (\ref{eqn:nil_proof_P_Q_R_matrices}), this holds as follows:
\begin{align*}
xy'PQ + yx'QP &= \frac{1}{2}(xy'PQ + yx'QP) + \frac{1}{2}(xy'PQ + yx'QP) \\
&= \frac{1}{2}(xy'PQ - xy'QP + xy'QP + yx'QP) + \frac{1}{2}(xy'PQ + yx'PQ - yx'PQ + yx'QP) \\
&= \frac{1}{2}(xy'R + xy'QP + yx'QP) + \frac{1}{2}(xy'PQ + yx'PQ - yx'R) \\
&= \frac{1}{2}(xy'-yx')R.
\end{align*}
Now, we have demonstrated that $\Phi$ is a group homomorphism. It remains to demonstrate that it is faithful. To see this, inspect the explicit image in (\ref{eqn:nil_proof_Phi_x_y_t_explicit}). If $\Phi(x,y,t) = I$, based on the off-diagonal entries, we have that $x = y = 0$. From this, using the diagonal entries, it follows that $1 \pm 2it\varepsilon^2 = 0$, from which it follows that $t = 0$. Thus, $\Phi$ is faithful. \\

In fact, if we projectivize, that is to say postcompose with the projection $\mathrm{SL}_2(\C[\varepsilon]/(\varepsilon^3)) \to \mathrm{PSL}_2(\C[\varepsilon]/(\varepsilon^3))$ to form a representation
\[
\widehat\Phi \colon \mathrm{Nil} \longrightarrow \mathrm{PSL}_2(\C[\varepsilon]/(\varepsilon^3))
\]
then it too is faithful. To see this, note that if $\Phi(x,y,t)$ is a scalar, then from the off-diagonal entries, we once again have $x = y = 0$, and then equality of the diagonal entries enforces $4it\varepsilon^2 = 0$ which forces $t = 0$. \\

We have now constructed a projective linear embedding of $\mathrm{Nil}$. Our goal is to construct such an embedding for $\mathrm{Isom}(\mathrm{Nil}) \cong \mathrm{Nil} \rtimes \mathrm{O}(2)$. To do so, we next construct a projective linear embedding for $\mathrm{O}(2)$. Any element of $\mathrm{O}(2)$ can be written as $R_\vartheta S^{\delta}$ where
\[
R_\vartheta =
\left(
\begin{array}{cc}
\cos\vartheta & -\sin\vartheta \\
\sin\vartheta & \cos\vartheta
\end{array}
\right)
,\quad
S =
\left(
\begin{array}{cc}
1 & 0 \\
0 & -1
\end{array}
\right)
\]
and where $\varepsilon$ is determined up to multiples of $2\pi$ and $\delta \in \{0,1\}$. Define matrices
\[
D_\vartheta =
\left(
\begin{array}{cc}
e^{i\vartheta/2} & 0 \\
0 & e^{-i\vartheta/2}
\end{array}
\right)
,\quad
J =
\left(
\begin{array}{cc}
0 & i \\
i & 0
\end{array}
\right).
\]
Now define a map
\[
\Theta \colon \mathrm{O}(2) \longrightarrow \mathrm{PSL}_2(\C[\varepsilon]/(\varepsilon^3))
\]
by setting $\Theta(R_\vartheta) = [D_\vartheta]$ and $\Theta(S) = [J]$. Note that a shift of $\vartheta$ by a multiple of $2\pi$ leads to a negation of $D_\vartheta$, which, in the projective group, leads to no change. Moreover, $J$ squares to $-I$, which amounts to $I$ in the projective group. The remaining verifications that $\Theta$ yields a group embedding are straightforward. \\

Finally, define a representation
\begin{equation}
\label{eqn:Pi_rep_of_isom_Nil}
\Pi \colon \mathrm{Nil} \rtimes \mathrm{O}(2) \longrightarrow \mathrm{PSL}_2(\C[\varepsilon]/(\varepsilon^3))
\end{equation}
by setting
\[
\Pi(n,A) := \widehat\Phi(n)\Theta(A).
\]
By the universal property of semidirect products, we know that this is a well-defined group map so long as, for any $n \in \mathrm{Nil}$ and $A \in \mathrm{O}(2)$, we have that
\[
\Theta(A)\widehat\Phi(n)\Theta(A)^{-1} = \widehat\Phi(A \cdot n).
\]
This follows from the following readily verifiable identities:
\[
\begin{array}{ll}
D_\vartheta P D_\vartheta^{-1} = \cos\vartheta P + \sin\vartheta\, Q & \quad J P J^{-1} = P \\
D_\vartheta Q D_\vartheta^{-1} = -\sin\vartheta P + \cos\vartheta\, Q & \quad J Q J^{-1} = -Q \\
D_\vartheta R D_\vartheta^{-1} = R & \quad J R J^{-1} = -R.
\end{array}
\]
Finally, we must check that $\Pi$ is faithful. Suppose that $\Pi(n, A) = 1$. If we reduce the matrices modulo $\varepsilon$, $\widehat\Phi(n)$ maps to $1$, whereas $\Theta(A)$ remains unchanged. As $\Theta$ is an embedding, we have that $A = I$. We are then left with $\widehat\Phi(n) = 1$, and so, because $\widehat\Phi$ is faithful, we have $n = 1$. This demonstrates that $\Pi$ is faithful.

\begin{Theorem}
\label{theorem:nil_proof}
If $M$ is an orientable connected compact $3$-manifold which admits a geometric structure modelled on $\mathrm{Nil}$, then $\pi_1M$ is residually $\mathrm{PGL}_2$, $\mathrm{PSL}_2$-finite.
\end{Theorem}

\begin{proof}
Given any such manifold $M$, $\pi_1M$ embeds into $\mathrm{Isom}(\mathrm{Nil})$. As in (\ref{eqn:Pi_rep_of_isom_Nil}) above, this isometry group embeds into $\mathrm{PSL}_2(\C[\varepsilon]/(\varepsilon^3))$. The conclusion now follows from Propositions \ref{prop:residual_K_finiteness_for_various_K} and \ref{prop:equivalent_characterizations}.
\end{proof}

\begin{Remark}
The above result is false if we replace commutative rings with fields, or equivalently with integral domains. To see this, note that the Heisenberg manifold $\mathrm{H}_3(\R)/\mathrm{H}_3(\Z)$, which admits $\mathrm{Nil}$ geometry, has fundamental group $\mathrm{H}_3(\Z)$ and then use Proposition 3.10 in \cite{FriedlGillTillmann} (the argument provided there invalidates residual $\mathrm{SL}_2$-finiteness; the case of residual $\mathrm{GL}_2$-finiteness follows by Proposition \ref{prop:residual_K_finiteness_for_various_K} and the observation that the abelianization of $\mathrm{H}_3(\Z)$ is free; the projective cases $\mathrm{PGL}_2$, $\mathrm{PSL}_2$ then also follow by choosing lifts and making certain trace and determinant arguments).
\end{Remark}

\section{The Geometry $\mathrm{Sol}$}

The fundamental groups of orientable connected compact $3$-manifolds which admit a geometric structure modelled on $\mathrm{Sol}$ are subgroups of the isometry group $\mathrm{Isom}^+(\mathrm{Sol})$. Here, $\mathrm{Sol}$ is the $3$-dimensional Lie group $\R^2 \rtimes \R$ where (see \cite[p. 470]{Scott}) the product is given by
\[
(x,y,t)  \cdot (x', y', t') = (x + e^{-t}x', y + e^ty', t + t').
\]
It is known (see, e.g., \cite{Scott}), that $\mathrm{Isom}(\mathrm{Sol})$ is isomorphic to the semi-direct product $\mathrm{Sol} \rtimes D_4$ where $D_4$, an isomorphic copy of the dihedral group of order eight, consists of the maps given by $(x,y,t) \mapsto (\pm x, \pm y, t)$ and $(x,y,t) \mapsto (\pm y, \pm x, -t)$. Of the eight maps in $D_4$, the orientation-preserving maps are
\begin{equation}
\label{equation:sol_proof_kappa_sigma}
1,\quad \kappa \colon (x,y,t) \mapsto (-x, -y, t),\quad \sigma \colon (x,y,t) \mapsto (y, x, -t),\quad \kappa\sigma \colon (x,y,t) \mapsto (-y,-x,-t).
\end{equation}
If we let $D_4^+$ denote the subgroup of $D_4$ generated by $\kappa$ and $\sigma$, it follows that $\mathrm{Isom}^+(\mathrm{Sol})$ is isomorphic to the semi-direct product $\mathrm{Sol} \rtimes D_4^+$. Note that $D_4^+ \cong (\Z / 2)^2$. We shall construct a linear embedding of $\mathrm{Isom}^+(\mathrm{Sol})$ by first constructing an embedding of $\mathrm{Sol}$. \\

Let $\varepsilon$ be an indeterminate, and consider the ring $\C[\varepsilon]/(\varepsilon^2)$. Set
\[
E = \left( \begin{array}{cc}
0 & 1 \\
0 & 0
\end{array} \right)
,\quad
F = \left( \begin{array}{cc}
0 & 0 \\
1 & 0
\end{array} \right).
\]
Here, the matrix entries lie in $\C[\varepsilon]/(\varepsilon^2)$. For $x, y \in \R$, define
\[
\mathrm{U}(x,y) := I + \varepsilon(xE + yF) =
\left(
\begin{array}{cc}
1 & x\varepsilon \\
y\varepsilon & 1
\end{array}
\right).
\]
Using $\varepsilon^2 = 0$, a straightforward check shows that $\det \mathrm{U}(x,y) = 1$ and that
\begin{equation}
\label{equation:sol_proof_U_additive}
\mathrm{U}(x, y)U(x', y') = \mathrm{U}(x + x', y + y').
\end{equation}
Next, for $t \in \R$, define
\[
\mathrm{A}(t) =
\left(
\begin{array}{cc}
e^{-t/2} & 0 \\
0 & e^{t/2}
\end{array}
\right).
\]
Again, the matrix entries lie in $\C[\varepsilon]/(\varepsilon^2)$. Note that $\det \mathrm{A}(t) = 1$, and that
\begin{equation}
\label{equation:sol_proof_A_additive}
\mathrm{A}(t)\mathrm{A}(t') = \mathrm{A}(t+t').
\end{equation}
Note also that
\begin{equation}
\label{equation:sol_proof_A_U_conjugation}
\mathrm{A}(t)\mathrm{U}(x,y)\mathrm{A}(t)^{-1} = \mathrm{U}(e^{-t}x, e^ty).
\end{equation}
Now define
\[
\Psi \colon \mathrm{Sol} \to \mathrm{SL}_2(\C[\varepsilon]/(\varepsilon^2))
\]
by
\[
\Psi(x,y,t) = \mathrm{U}(x,y)\mathrm{A}(t).
\]
Using (\ref{equation:sol_proof_U_additive}), (\ref{equation:sol_proof_A_additive}) and (\ref{equation:sol_proof_A_U_conjugation}), we can see that $\Psi$ is a homomorphism as follows:
\begin{align*}
\Psi(x,y,t)\Psi(x',y',t') &= \mathrm{U}(x,y)\mathrm{A}(t)\mathrm{U}(x',y')\mathrm{A}(t') \\
&= \mathrm{U}(x,y)\mathrm{A}(t)\mathrm{U}(x',y')\mathrm{A}(t)^{-1}\mathrm{A}(t)\mathrm{A}(t') \\
&= \mathrm{U}(x,y)\mathrm{U}(e^{-t}x', e^ty')\mathrm{A}(t+t') \\
&= \mathrm{U}(x+e^{-t}x', y+e^ty')\mathrm{A}(t+t') \\
&= \Psi((x,y,t) \cdot (x',y',t')).
\end{align*}
Moreover, writing out $\Psi(x,y,t)$ explicitly, we find that
\[
\Psi(x,y,t) =
\left(
\begin{array}{cc}
e^{-t/2} & xe^{t/2}\varepsilon \\
ye^{-t/2}\varepsilon & e^{t/2}
\end{array}
\right).
\]
From this, it follows immediately that $\Psi$ is faithful. In fact, if we projectivize, that is to say postcompose with the projection $\mathrm{SL}_2(\C[\varepsilon]/(\varepsilon^2)) \to \mathrm{PSL}_2(\C[\varepsilon]/(\varepsilon^2))$ to form a representation
\[
\widehat\Psi \colon \mathrm{Sol} \longrightarrow \mathrm{PSL}_2(\C[\varepsilon]/(\varepsilon^2))
\]
then it too is faithful. To see this, note that if $\Psi(x,y,t)$ is a scalar, then from the off-diagonal entries, we have $x = y = 0$, and then equality of the diagonal entries enforces $e^t = 1$ which forces $t = 0$. \\

We have now constructed a projective linear embedding of $\mathrm{Sol}$. Our goal is to construct such an embedding for $\mathrm{Isom}^+(\mathrm{Sol}) \cong \mathrm{Sol} \rtimes D_4^+$. To do so, we next construct a projective linear embedding for $D_4^+$. This group is generated by the isometries $\kappa$, $\sigma$ in (\ref{equation:sol_proof_kappa_sigma}). Set
\[
K =
\left(
\begin{array}{cc}
i & 0 \\
0 & -i
\end{array}
\right)
,\quad
S =
\left(
\begin{array}{cc}
0 & i \\
i & 0
\end{array}
\right)
\]
and then define
\[
\Omega \colon D_4^+ \to \mathrm{PSL}_2(\C[\varepsilon]/(\varepsilon^2))
\]
by setting
\[
\kappa \mapsto [K]
,\quad
\sigma \mapsto [S].
\]
A straightforward check shows that $\Omega$ is a faithful representation. \\

Finally, define a representation
\begin{equation}
\label{eqn:Sigma_rep_of_isom_Sol}
\Sigma \colon \mathrm{Sol} \rtimes D_4^+ \longrightarrow \mathrm{PSL}_2(\C[\varepsilon]/(\varepsilon^2))
\end{equation}
by setting
\[
\Sigma(p,\alpha) := \widehat\Psi(p)\Omega(\alpha).
\]
By the universal property of semidirect products, we know that this is a well-defined group map so long as, for any $p \in \mathrm{Sol}$ and $\alpha \in D_4^+$, we have that
\[
\Omega(\alpha)\widehat\Psi(p)\Omega(\alpha)^{-1} = \widehat\Psi(\alpha \cdot p).
\]
This follows from the following readily verifiable identities:
\[
\begin{array}{ll}
K E K^{-1} = -E & \quad S E S^{-1} = F \\
K F K^{-1} = -F & \quad S F S^{-1} = E \\
K A(t) K^{-1} = A(t) & \quad S A(t) S^{-1} = A(-t).
\end{array}
\]
Finally, we must check that $\Sigma$ is faithful. Suppose that $\Sigma(p, \alpha) = 1$. If we reduce the matrices modulo $\varepsilon$, $\widehat\Psi(p)$ maps to $A(t)$, whereas $\Omega(\alpha)$ remains unchanged. A simple check shows that $A(t)\Omega(\alpha) = 1$ forces $\Omega(\alpha) = 1$. As $\Omega$ is an embedding, we have that $\alpha = 1$. We are then left with $\widehat\Psi(p) = 1$, and so, because $\widehat\Psi$ is faithful, we have $p = 1$. This demonstrates that $\Sigma$ is faithful.

\begin{Theorem}
\label{theorem:nil_proof}
If $M$ is an orientable connected compact $3$-manifold which admits a geometric structure modelled on $\mathrm{Sol}$, then $\pi_1M$ is residually $\mathrm{PGL}_2$, $\mathrm{PSL}_2$-finite.
\end{Theorem}

\begin{proof}
Given any such manifold $M$, $\pi_1M$ embeds into $\mathrm{Isom}^+(\mathrm{Sol})$. As in (\ref{eqn:Sigma_rep_of_isom_Sol}) above, this isometry group embeds into $\mathrm{PSL}_2(\C[\varepsilon]/(\varepsilon^2)$. The conclusion now follows from Propositions \ref{prop:residual_K_finiteness_for_various_K} and \ref{prop:equivalent_characterizations}.
\end{proof}

\section{Spherical Geometry $\Sph^3$}

In the case of spherical $3$-manifolds, we have the following.

\begin{Theorem}
\label{thm:spherical_manifolds}
If $M$ is an orientable connected compact 3-manifold and admits a geometric structure modelled on $\Sph^3$, then $\pi_1(M)$ is residually $\mathrm{GL}_2, \mathrm{SL}_2$-finite.
\end{Theorem}

\begin{proof}
As per \cite[Theorem 1.2]{KodaniWatanabe}, $\pi_1M$ is isomorphic to one of the following:
\begin{itemize}
	\item[(1)] A cyclic group $\Z_n$.
	\item[(2)] A direct product $\Z_m \times P$ where $P$ is a binary polyhedral group (binary dihedral, tetrahedral, octahedral, or icosahedral).
	\item[(3)] A direct product $\Z_m \times D'_{2^{k+2}p} = \langle x, y \, \mid \, x^{2^{k+2}} = 1, y^p = 1, xy^{-1} = yx \rangle$, $m > 0$, $k \ge 0$, $p \ge 3$ odd, and $(m, 2p) = 1$.
	\item[(4)] A direct product $\Z_m \times T'_{8 \cdot 3^k} = \langle x, y, z \, \mid \, x^2 = (xy)^2 = y^2, zxz^{-1} = y, zyz^{-1} = xy, z^{3^k} = 1 \rangle$, $m > 0$, $(m, 6) = 1$, and $k \ge 1$.
\end{itemize}
By Proposition \ref{prop:residual_finiteness_for_abelian_groups}, each cyclic group is residually $\mathrm{GL}_2, \mathrm{SL}_2$-finite. Moreover, as in \cite[Chapter 7]{Wolf}, every binary polyhedral group is a subgroup of $\mathrm{SU}(2)$, and so admits a faithful representation into $\mathrm{SL}(2, \C)$. The result, for cases (1) and (2), now follows by Propositions \ref{prop:equivalent_characterizations} and \ref{prop:products_preserve_residual_k_finiteness}. For case (3), as in \cite[Section 4.3]{KodaniWatanabe}, let
\[
D^*_{4p} = \langle r, s \, \mid \, r^{2p} = 1, s^2 = r^p, srs^{-1} = r^{-1} \rangle
\]
be the binary dihedral group of order $4p$. One can check that
\[
x \mapsto (s, 1) \quad y \mapsto (r^2, 0)
\]
defines an embedding $D'_{2^{k+2}p} \hookrightarrow D^*_{4p} \times \Z_{2^{k+2}}$. For case (4), as in \cite[Section 4.5]{KodaniWatanabe}, let
\[
T^* = \langle r,s,t \, \mid \, r^2 = (rs)^2 = s^2, trt^{-1} = s, tst^{-1} = rs, t^3 = 1\rangle
\]
be the binary tetrahedral group of order 24. One can check that
\[
x \mapsto (r,0) \quad y \mapsto (s,0) \quad z \mapsto (t,1)
\]
defines an embedding $T'_{8 \cdot 3^k} \hookrightarrow T^* \times \Z_{3^k}$. The result, for cases (3) and (4), now follows by Propositions \ref{prop:equivalent_characterizations} and \ref{prop:products_preserve_residual_k_finiteness}.
\end{proof}

The above result does not demonstrate residual $\mathrm{PSL}_2$, $\mathrm{PGL}_2$-finiteness for spherical 3-manifold groups. In fact, these properties do not hold for such groups, as we now demonstrate.

\begin{Theorem}
\label{thm:spherical_counterexample}
Let $P$ denote the Poincar\'{e} homology sphere $\Sigma(2,3,5)$, which is an orientable connected compact 3-manifold, and which admits a geometric structure modelled on $\Sph^3$. The fundamental group $\pi_1P$ is not residually $\mathrm{K}$-finite for any of $\mathrm{K} = \mathrm{PGL}_2, \mathrm{PSL}_2$.
\end{Theorem}

\begin{proof}
It is known that $\pi_1P \cong \mathrm{SL}(2,5)$, which is a perfect group. The element $-I$ is central. By Proposition \ref{prop:perfect_group_counterexample}, $\mathrm{SL}(2,5)$ cannot be residually $\mathrm{PGL}_2$-finite at $-I$. By Proposition \ref{prop:residual_K_finiteness_for_various_K}, it then also cannot be residually $\mathrm{PSL}_2$-finite at $-I$.
\end{proof}

\begin{Remark}
If $M$ is an orientable connected compact 3-manifold which admits a geometric structure modelled on $\Sph^3$, the holonomy representation arising from the geometric structure gives an embedding $\mathrm{hol} \colon \pi_1M \hookrightarrow \mathrm{Isom}^+(\Sph^3) \cong \text{SO}(4,\R)$. Moreover, as in \cite[Chapter 2]{Stillwell} and \cite[Theorem 9.4.1]{Artin}, we have that 
\[
\mathrm{SO}(4,\R)/\{\pm I\} \cong \mathrm{SO}(3,\R) \times \mathrm{SO}(3,\R) \cong \mathrm{PSU}(2,\C) \times \mathrm{PSU}(2,\C).
\]
Combining the above with Propositions \ref{prop:residual_K_finiteness_for_various_K} and \ref{prop:equivalent_characterizations}, we have residual $\mathrm{GL}_2$, $\mathrm{SL}_2$, $\mathrm{PGL}_2$, $\mathrm{PSL}_2$-finiteness for $\pi_1M$ at all elements except possibly at $\mathrm{hol}^{-1}(-I)$. It is thus not surprising to see $-I$ arise in the counterexample provided in Theorem \ref{thm:spherical_counterexample}.
\end{Remark}

\section{The Geometry $\widetilde{\mathrm{SL}_2}$}
\label{section:SL_2_tilde}

The orientable connected compact $3$-manifolds which admit a geometric structure modeled on $\widetilde{\mathrm{SL}_2}$ are known to be the total spaces of certain Seifert bundles. Recall, as per \cite{Seifert, Orlik, Scott}, that a Seifert bundle $M \to S$ is characterized by an invariant of the form
\[
(b,(\varepsilon, g); (a_1,b_1), \dots, (a_n,b_n)).
\]
Here:
\begin{itemize}
\item $b$ is an integer.
\item $\varepsilon = +1$ if $S$ is orientable and $\varepsilon = -1$ if $S$ is non-orientable.
\item $g$ is the genus of $S$; $g \ge 1$ if $\varepsilon = -1$.
\item $n \in \N$ and the $(a_i,b_i)$ are pairs of integers determining the type of each of the $n$ exceptional orbits, where, for each $i \in \{1, \dots, n\}$, $a_i$ and $b_i$ are coprime and $0 < b_i < a_i$.
\end{itemize}

Given such a bundle $M \to S$, with invariant $(b,(\varepsilon, g); (a_1,b_1), \dots, (a_n,b_n))$, recall also that, if $\varepsilon = +1$, $\pi_1M$ admits the  presentation
\begin{equation}
\label{eqn:seifert_bundle_orientable_base_pi1_presentation}
\pi_1M =
\left\langle
\begin{array}{c}
h \\
s_1,\ldots,s_n \\
u_1,v_1,\ldots,u_g,v_g
\end{array}
\left|
\begin{array}{cl}
[h,s_i] = [h,u_j] = [h,v_j] = 1, &1 \le i \le n, \: 1 \le j \le g \\
s_i^{a_i}h^{b_i} = 1, &1 \le i \le n \\
(s_1 \cdots s_n)[u_1,v_1] \cdots [u_g,v_g] = h^b &
\end{array}
\right.
\right\rangle
\end{equation}
and, if $\varepsilon = -1$, $\pi_1M$ admits instead the  presentation
\begin{equation}
\label{eqn:seifert_bundle_nonorientable_base_pi1_presentation}
\pi_1M =
\left\langle
\begin{array}{c}
h \\
s_1,\ldots,s_n \\
v_1,\ldots,v_g
\end{array}
\left|
\begin{array}{cl}
[h,s_i] = 1, \quad v_jhv_j^{-1} = h^{-1}, &1 \le i \le n, \: 1 \le j \le g \\
s_i^{a_i}h^{b_i} = 1, &1 \le i \le n \\
(s_1 \cdots s_n)(v_1^2 \cdots v_g^2) = h^b &
\end{array}
\right.
\right\rangle.
\end{equation}

Now, as per, e.g., \cite[Corollary 12.6.6]{Martelli}, in the case of $3$-manifolds $M$ which admit a geometric structure modelled on $\widetilde{\mathrm{SL}_2}$, the Seifert bundles which arise are exactly those where the following conditions hold:
\begin{itemize}
\item The Euler orbifold characteristic $\chi(M) := \chi(S) - \sum\left(1-\frac{1}{a_i}\right)$ is negative. Here $\chi(S)$ is $2-2g$ if $\varepsilon = +1$, and $2-g$ if $\varepsilon = -1$.
\item The circle bundle Euler number $e(M) := b + \sum\frac{b_i}{a_i}$ is non-zero. Note that then $e(M)$ is strictly positive for one of the two possible orientations -- let us always choose this orientation.
\end{itemize}

Now, before proceeding further, we need a technical lemma.

\begin{Lemma}
\label{lem:SL2_tilde_proof_ideal_lemma}
Let $p^{\pm}$ and $x$ be indeterminates and set $R$ to be the ring $\C[p^{\pm}, x]$. Fix positive integers $A$, $B$, and set $I$ to be the ideal generated by $(p^A-p^{-A})x$, $x^3$ and $1+x^2-p^B$. For each non-zero $n$, $1 - p^n \notin I$.
\end{Lemma}

\begin{proof}
It suffices to demonstrate the conclusion for positive integers $n$. Suppose, by way of contradiction, that there exist polynomials $F_1(p,x)$, $F_2(p,x)$, $F_3(p,x)$ such that
\begin{equation}
\label{eqn:SL2_tilde_proof_ideal_lemma}
F_1(p,x)(p^A-p^{-A})x + F_2(p,x)x^3 + F_3(p,x)(1+x^2-p^B) = 1-p^n.
\end{equation}
Let $F_3(p,x) = G_0(p) + G_1(p)x + \cdots + G_d(p)x^d$, for some degree $d$ and polynomials $G_i$ in only the indeterminate $p$. Upon substituting $x = 0$ into (\ref{eqn:SL2_tilde_proof_ideal_lemma}), we find that
\begin{equation}
\label{eqn:SL2_tilde_proof_ideal_lemma_2}
G_0(p)(1-p^B) = 1 - p^n.
\end{equation}
Note that this forces $G_0(p)$ to contain only non-negative powers of $p$. Next, upon substitution of $p = 1$ into (\ref{eqn:SL2_tilde_proof_ideal_lemma}), we find that $F_2(1,x)x^3 + F_3(1,x)x^2 = 0$, and so $F_2(1,x)x + F_3(1,x) = 0$. Upon substituting $x=0$, we find that $G_0(1) = 0$. Combining this with (\ref{eqn:SL2_tilde_proof_ideal_lemma_2}), we have a contradiction: $G_0(p)(1-p^B)$ then has a root at $p = 1$ with multiplicity at least $2$, whereas $1 - p^n$ has only a simple root at $p = 1$. This completes the proof.
\end{proof}

\begin{Theorem}
If $M$ is an orientable connected compact $3$-manifold which admits a geometric structure modelled on $\widetilde{\mathrm{SL}_2}$, then $M$ is the total space of a Seifert bundle $M \to S$. If $S$ is orientable and has genus $g \ge 1$, or if $S$ is non-orientable and has genus $g \ge 3$, we have that $\pi_1M$ is residually $\mathrm{PGL}_2$, $\mathrm{PSL}_2$-finite.
\end{Theorem}

\begin{proof}
As per \cite[Proposition 12.6.2]{Martelli}, we have a short exact sequence
\[
0 \longrightarrow \R \longrightarrow \mathrm{Isom}^+(\widetilde{\mathrm{SL}_2}) \overset{p}{\longrightarrow} \mathrm{Isom}(\Hyp^2) \longrightarrow 0.
\]
By \cite[Proposition 8.5]{Iversen}, we have that $\mathrm{Isom}(\Hyp^2) \cong \mathrm{PGL}(2,\R)$.
Under the holonomy representation, $\pi_1M$ injects into $\mathrm{Isom}^+(\widetilde{\mathrm{SL}_2})$, and, under the presentations given above, the image of this injection intersects the kernel of $p$ in exactly the subgroup generated by the regular fibre $h$. This, combined with Propositions \ref{prop:residual_K_finiteness_for_various_K} and \ref{prop:equivalent_characterizations}, shows that $\pi_1M$ is residually $\mathrm{PGL}_2$, $\mathrm{PSL}_2$-finite at all elements outside of the subgroup generated by $h$. It remains to consider the powers of $h$. \\

Consider first the case where the base $S$ of the Seifert bundle is orientable. In this case, we have the presentation as in (\ref{eqn:seifert_bundle_orientable_base_pi1_presentation}) above, where $g \ge 1$ and $e = b + \sum \frac{b_i}{a_i} > 0$. Let $p^{\pm}$ and $x$ be indeterminates, set $A := a_1 \cdots a_n$ and consider the following putative presentation of $\pi_1M$:
\[
h \mapsto \left( \begin{array}{cc}
p^A & 0 \\
0 & p^{-A}
\end{array} \right)
,\:
s_i \mapsto \left( \begin{array}{cc}
p^{-\frac{b_i}{a_i}A} & 0 \\
0 & p^{\frac{b_i}{a_i}A}
\end{array} \right)
,\:
u_1 \mapsto \left( \begin{array}{cc}
1 & x \\
0 & 1
\end{array} \right)
,\:
v_1 \mapsto \left( \begin{array}{cc}
1 & 0 \\
x & 1
\end{array} \right).
\]
For $j \ge 2$, $u_j$ and $v_j$ are to be mapped to the identity. By construction, we have satisfied the relations $[h,s_i] = 1$ and $s_i^{a_i}h^{b_i} = 1$ for $1 \le i \le n$, and also $[h,u_j] = [h,v_j] = 1$ for $j \ge 2$. In order to satisfy $[h,u_1] = [h,v_1] = 1$, we need to enforce the relation $(p^A-p^{-A})x = 0$. In order to satisfy the relation $(s_1 \cdots s_n)[u_1,v_1] \cdots [u_g,v_g] = h^b$, or equivalently, $[u_1,v_1] = s_n^{-1} \cdots s_1^{-1}h^b$, we compute both sides to find that
\[
s_n^{-1} \cdots s_1^{-1}h^b = \left( \begin{array}{cc}
p^{Ae} & 0 \\
0 & p^{-Ae}
\end{array} \right)
,\:
[u_1,v_1] = \left( \begin{array}{cc}
1+x^2+x^4 & -x^3 \\
x^3 & 1-x^2
\end{array} \right).
\]
It follows that this relation is satisfied if we enforce $x^3 = 0$ and $1+x^2-p^{Ae} = 0$. As such, we have a well defined representation $\pi_1M \to \mathrm{SL}(2,R)$ where $R := \C[p^{\pm},x]/I$, where $I$ is the ideal generated by the polynomials $(p^A-p^{-A})x$, $x^3$, $1+x^2-p^{Ae}$. By Lemma~\ref{lem:SL2_tilde_proof_ideal_lemma}, we have that, for each non-zero $n$, $1-p^n \notin I$. It follows that all non-trivial powers of $h$ have a non-trivial image. Moreover, such an image cannot be a scalar matrix, for $p^{An} = p^{-An}$ implies $p^{2An} = 1$. It follows that, at the powers of the central element $h$, $\pi_1M$ is residually $\mathrm{PSL}_2$-finite, and so also $\mathrm{PGL}_2$-finite, as desired. \\

Now consider the case where the base $S$ of the Seifert bundle is non-orientable. In this case, we have the presentation as in (\ref{eqn:seifert_bundle_nonorientable_base_pi1_presentation}) above, where $g \ge 3$ and $e = b + \sum \frac{b_i}{a_i} > 0$. The above approach appears at first to not work in this case, because of the lack of a commutator in the relation $(s_1 \cdots s_n)(v_1^2 \cdots v_g^2) = h^b$. However, recall that, by Dyck's theorem, a non-orientable surface with 
at least three crosscaps must contain a handle. Algebraically, we can see this with the following transformation. Leave $h$, the $s_i$ for all $i$, and the $v_j$ for $j \ge 4$ as they are. Replace $v_1$, $v_2$ and $v_3$ by:
\[
a = v_3^{-1}v_2^{-1} \qquad b = v_1v_2 \qquad c = v_1v_2v_3.
\]
This is a Nielsen automorphism of $\langle v_1, v_2, v_3 \rangle$; the inverse is given by
\[
v_1 = ca \qquad v_2 = a^{-1}c^{-1}b \qquad v_3 = b^{-1}c.
\]
A straightforward verification shows that, with these new generators, $\pi_1M$ admits the following presentation
\[
\pi_1M =
\left\langle
\begin{array}{c}
h \\
s_1,\ldots,s_n \\
a, b, c, v_4,\ldots,v_g
\end{array}
\left|
\begin{array}{cl}
[h,s_i] = [h,a] = [h,b] = 1, \quad chc^{-1} = v_jhv_j^{-1} = h^{-1}, &1 \le i \le n, \: 4 \le j \le g \\
s_i^{a_i}h^{b_i} = 1, &1 \le i \le n \\
(s_1 \cdots s_n)(c[a,b]c)(v_4^2 \cdots v_g^2) = h^b &
\end{array}
\right.
\right\rangle.
\]
Note that $c[a,b]c$ is a conjugate of $[a,b]c^2$, the standard handle-plus-crosscap relator; geometrically, we've traded three crosscaps for one handle $[a,b]$ and one crosscap $c$. \\

Now we proceed as follows. As before, let $p^{\pm}$ and $x$ be indeterminates, set $A := a_1 \cdots a_n$ and consider the following putative presentation of $\pi_1M$:
\[
h \mapsto \left( \begin{array}{cc}
p^A & 0 \\
0 & p^{-A}
\end{array} \right)
,\:
s_i \mapsto \left( \begin{array}{cc}
p^{-\frac{b_i}{a_i}A} & 0 \\
0 & p^{\frac{b_i}{a_i}A}
\end{array} \right)
,\:
a \mapsto \left( \begin{array}{cc}
1 & 0 \\
x & 1
\end{array} \right)
,\:
b \mapsto \left( \begin{array}{cc}
1 & x \\
0 & 1
\end{array} \right)
,\:
c \mapsto \left( \begin{array}{cc}
0 & 1 \\
1 & 0
\end{array} \right).
\]
(Note that interchange of the matrices for $a$ and $b$, as opposed to $u_1$ and $v_1$ above.) For $j \ge 4$, $v_j$ is sent to the same swap matrix as $c$ above. By construction, we have satisfied the relations $[h,s_i] = 1$ and $s_i^{a_i}h^{b_i} = 1$ for $1 \le i \le n$, and also $chc^{-1} = h^{-1}$ and $v_jhv_j^{-1} = h^{-1}$ for $j \ge 4$. In order to satisfy $[h,a] = [h,b] = 1$, we need to enforce the relation $(p^A-p^{-A})x = 0$. In order to satisfy the relation $(s_1 \cdots s_n)(c[a,b]c)(v_4^2 \cdots v_g^2) = h^b$, or equivalently, $c[a,b]c = s_n^{-1} \cdots s_1^{-1}h^b$, we compute both sides to find that
\[
s_n^{-1} \cdots s_1^{-1}h^b = \left( \begin{array}{cc}
p^{Ae} & 0 \\
0 & p^{-Ae}
\end{array} \right)
,\:
c[a,b]c = \left( \begin{array}{cc}
1+x^2+x^4 & -x^3 \\
x^3 & 1-x^2
\end{array} \right).
\]
It follows that this relation is satisfied if we enforce $x^3 = 0$ and $1+x^2-p^{Ae} = 0$. As such, we have a well defined representation $\pi_1M \to \mathrm{GL}(2,R)$ where $R := \C[p^{\pm},x]/I$, where $I$ is the ideal generated by the polynomials $(p^A-p^{-A})x$, $x^3$, $1+x^2-p^{Ae}$. Note that, in this case, the target group is the general linear group, not the special linear group. By the same argument as in the case where the base $S$ of the Seifert bundle is orientable, we have residual $\mathrm{PGL}_2$-finiteness, and residual $\mathrm{PSL}_2$-finiteness then also follows by Proposition \ref{prop:residual_K_finiteness_for_various_K}.
\end{proof}

The above result demonstrates residual $\mathrm{PGL}_2$, $\mathrm{PSL}_2$-finiteness of some 3-manifolds with $\widetilde{\mathrm{SL}_2}$ geometry. As the following result shows, the result does not hold for all such 3-manifolds.

\begin{Theorem}
\label{thm:SL2_tilde_counterexample}
Let $B$ denote the Brieskorn sphere $\Sigma(2,3,7)$, which is an orientable connected compact 3-manifold, and which admits a geometric structure modelled on $\widetilde{\mathrm{SL}_2}$. The fundamental group $\pi_1B$ is not residually $\mathrm{K}$-finite for any of $\mathrm{K} = \mathrm{GL}_2, \mathrm{SL}_2, \mathrm{PGL}_2, \mathrm{PSL}_2$.
\end{Theorem}

\begin{proof}
As $B$ is an integral homology sphere, $\pi_1 B$ is perfect. Moreover, the regular fibre $h$ generates an infinite cyclic centre. By Proposition \ref{prop:perfect_group_counterexample}, if $R$ is a finite commutative ring, every homomorphism $\pi_1B \to \mathrm{GL}(2, R)$ kills $h^2$, and every homomorphism $\pi_1B \to \mathrm{PGL}(2, R)$ kills $h$. Thus $\pi_1B$ is not residually $\mathrm{GL}_2$-finite, nor residually $\mathrm{PGL}_2$-finite, and so of course also not residually $\mathrm{SL}_2$-finite, nor residually $\mathrm{PSL}_2$-finite.
\end{proof}

\end{document}